\documentclass[10pt]{article}

\usepackage[margin=1.12in]{geometry}
\usepackage{amsmath,amssymb,amsthm,mathtools}
\usepackage{authblk}
\usepackage{hyperref}
\usepackage{microtype}

\hypersetup{colorlinks=true,linkcolor=blue,citecolor=blue,urlcolor=blue}

\newtheorem{theorem}{Theorem}[section]
\newtheorem{lemma}[theorem]{Lemma}
\newtheorem{proposition}[theorem]{Proposition}
\newtheorem{corollary}[theorem]{Corollary}
\newtheorem{conjecture}[theorem]{Conjecture}
\newtheorem{question}[theorem]{Question}

\newcommand{\ex}{\operatorname{ex}}

\newcommand{\floor}[1]{\left\lfloor#1\right\rfloor}

\title{\large\bf Strong counterexamples to a supersaturation question of Ma--Yuan}
\author{
Wanfang~Chen\thanks{School of Mathematical Sciences, University of Science and Technology of China, Hefei, 230026, P.R.  China.
        Email: \texttt{a372959313@gmail.com.}}\qquad
Long-Tu~Yuan\thanks{School of Mathematical Sciences and Shanghai Key Laboratory of PMMP, East China Normal University, 500 Dongchuan Road, Shanghai 200240, P.R.  China.
        Email: \texttt{ltyuan@math.ecnu.edu.cn.} 
        Supported in part by National Natural Science Foundation of China grant 12271169 and Science and Technology Commission
        of Shanghai Municipality (No. 22DZ2229014).
    }}
\date{\today}

\begin{document}
\maketitle

\begin{abstract}
For a graph $F$, let $h_F(n,q)$ be the minimum number of copies of $F$ in an $n$-vertex graph with $\ex(n,F)+q$ edges, where $\ex(n,F)$ is the maximum number of edges in an $n$-vertex $F$-free graph. Let $c(n,F)$ be the minimum number of copies obtained by adding one edge to an extremal $F$-free graph. Mubayi's supersaturation conjecture predicts, under a stability hypothesis, that $h_F(n,q)\ge q\,c(n,F)$. Ma and Yuan recently constructed stable graph counterexamples for every fixed $q\ge4$; they asked whether the one-edge equality $h_F(n,1)=c(n,F)$ might still hold for every graph $F$ containing a cycle.

We give a negative answer to their question. For each integer $t\ge6$, let $H_t$ be obtained from the $t$-vertex path by replacing each edge with a $3t$-page book, using disjoint page vertices for different path edges. Then $h_{H_t}(n,1)<c(n,H_t)$ for infinitely many values of $n$. Moreover, by taking $t$ large, the ratio $h_{H_t}(n,1)/c(n,H_t)$ can be made arbitrarily small along infinitely many values of $n$.
\end{abstract}

\section{Introduction}

Let $F$ be a graph. Write $\ex(n,F)$ for the maximum number of edges in an $n$-vertex $F$-free graph. Let $N_F(G)$ denote the number of copies of $F$ in $G$, counted as subgraphs. For $q\ge1$, define
\[
   h_F(n,q)\coloneqq\min\{N_F(G): |V(G)|=n,\ e(G)=\ex(n,F)+q\}.
\]
Following Ma and Yuan~\cite{MaYuan2023}, let $t_F(n,q)$ be the minimum number of copies of $F$ in a graph obtained from an $n$-vertex extremal $F$-free graph by adding $q$ new edges. In particular,
\[
   c(n,F)\coloneqq t_F(n,1)
\]
is the usual one-edge comparison quantity. Clearly $h_F(n,q)\le t_F(n,q)$.

For cliques, the classical results of Rademacher (unpublished; see Erd\H{o}s~\cite{Erdos1962}), Erd\H{o}s~\cite{Erdos1955,Erdos1962}, and Lov\'asz--Simonovits~\cite{LovaszSimonovits1976,LovaszSimonovits1983} show that the minimum supersaturation just above the Tur\'an number is obtained by adding the extra edges to an extremal graph. Mubayi extended this phenomenon to color-critical graphs~\cite{Mubayi2010}, and Pikhurko and Yilma later proved the exact equality $h_F(n,q)=t_F(n,q)$ for color-critical graphs in the range $1\le q\le \varepsilon_F n$~\cite{PikhurkoYilma2017}. Motivated by the stability method, Mubayi formulated the following conjecture for hypergraphs (including graphs)~\cite{Mubayi2013}.

\begin{conjecture}[\cite{Mubayi2013}]\label{conj:mubayi}
Let $r\ge2$ and let $F$ be a non-$r$-partite stable $r$-graph. For every positive integer $q$, if $n$ is sufficiently large, then
\[
   h_F(n,q)\ge q\,c(n,F).
\]
\end{conjecture}

Here stable means that $\ex(n,F)$ is achieved by a unique $n$-vertex $r$-graph for all sufficiently large $n$, and every $F$-free $r$-graph with $(1-o(1))\ex(n,F)$ edges is $o(n^r)$-close to this extremal graph.

Ma and Yuan~\cite{MaYuan2023} refuted Conjecture~\ref{conj:mubayi} in the graph case by constructing stable non-bipartite graphs for which the inequality fails for every fixed $q\ge4$. Their examples are delicate, and they asked whether the one-edge case might still be true.

\begin{question}[\cite{MaYuan2023}]\label{q:mayuan}
Is it true that for every graph $F$ containing a cycle and for all sufficiently large $n$, one has
\[
   h_F(n,1)=t_F(n,1)?
\]
\end{question}

This question does not impose the stability hypothesis from Conjecture~\ref{conj:mubayi}.  We give a strong negative answer.

Our counterexamples are book expansions of paths.  Let $P_t$ be the path $v_1v_2\cdots v_t$ on $t$ vertices.  For $t\ge6$, define $H_t$ as follows.  For each edge $v_iv_{i+1}$ of $P_t$, add $3t$ new vertices adjacent to both $v_i$ and $v_{i+1}$, with these new vertices chosen disjointly for different edges.  Equivalently, each edge of the core path is the spine of a $3t$-page book.  In symbols,
\[
   H_t\coloneqq P_t^{(3t)}.
\]
Since $H_t$ contains triangles, it is non-bipartite.

The first theorem gives the quantitative one-edge comparison needed for the construction.
\begin{theorem}\label{thm:main}
Fix $t\ge6$. There is $s_0=s_0(t)$ such that for every $s\ge s_0$, one has
\begin{enumerate}
    \item $h_{H_t}\left( 2(t-1)s,1 \right) \le \frac{t!}{2}\Lambda_{t,s}$, 
    \item $c\left( 2(t-1)s,H_t \right) \ge 2^{t-2}(t-2)!\Lambda_{t,s}$, 
\end{enumerate}
where $\Lambda_{t,s}\coloneqq\prod_{j=0}^{t-2}\binom{(t-1)s-3tj}{3t}$. 
Consequently,
\[
   \frac{h_{H_t}(2(t-1)s,1)}{c(2(t-1)s,H_t)}
   \le
   \frac{(t!/2)\Lambda_{t,s}}{2^{t-2}(t-2)!\Lambda_{t,s}}
   =
   \frac{t(t-1)}{2^{t-1}}
   \to 0 \quad\text{as}\quad t \to \infty. 
\]
\end{theorem}

The rest of this note is organized as follows.  Section~\ref{sec:extremal} proves the extremal theorem for $H_t$, and Section~\ref{sec:one-edge-count} gives the one-edge count.

\section{The extremal graphs}\label{sec:extremal}

We use the exact path theorem in the following form.  If $N=(t-1)q+r$, $0\le r<t-1$, then
\begin{equation}\label{eq:path-exact}
   p_t(N)\coloneqq \ex(N,P_t)=q\binom{t-1}{2}+\binom{r}{2}.
\end{equation}
Moreover, when $N$ is divisible by $t-1$, the unique extremal $P_t$-free graph on $N$ vertices is $(N/(t-1))K_{t-1}$.  This is the exact Erd\H{o}s--Gallai path theorem due to Faudree--Schelp and Kopylov~\cite{ErdosGallai1959,FaudreeSchelp1975,Kopylov1977}.

Let $B_R$ be the $R$-page book, that is, $R$ triangles with a common edge.  We also use the theorem of Edwards and Khad\v{z}iivanov--Nikiforov that every $n$-vertex graph with more than $\floor{n^2/4}$ edges has an edge contained in at least $n/6$ triangles~\cite{Edwards1975,KhadziivanovNikiforov1979}.  Hence, for every fixed $R$, a $B_R$-free graph has at most $\floor{n^2/4}$ edges for all sufficiently large $n$.

Finally, we use the following consequence of Theorem~1.9 of Miao, Liu, and van Dam~\cite{MiaoLiuVanDam2025}: for every fixed $R$, every non-bipartite $B_R$-free graph on $n$ vertices has at most
\begin{equation}\label{eq:nonbip-book}
   \left\lfloor \frac{(n-1)^2}{4}\right\rfloor+O_R(1)
\end{equation}
edges.  Thus a $B_R$-free graph with at least $\floor{n^2/4}-O_R(1)$ edges is bipartite when $n$ is large enough.

Put
\[
   M_t(n)\coloneqq\max_{a+b=n}\{ab+p_t(a)+p_t(b)\}.
\]
Let $\mathcal S_{t,n}$ be the following structural family.  A graph $G$ belongs to $\mathcal S_{t,n}$ if there is a partition $V(G)=A\cup B$ such that
\[
   e(G)=|A||B|+p_t(|A|)+p_t(|B|)=M_t(n),
\]
all edges between $A$ and $B$ are present, and $G[A]$ and $G[B]$ are extremal $P_t$-free graphs.  Let $\mathcal G_{t,n}$ be the subfamily of $H_t$-free graphs in $\mathcal S_{t,n}$.

The next proposition is the extremal input for the counting argument.
\begin{proposition}\label{prop:extremal}
Fix $t\ge6$.  For all sufficiently large $n$,
\[
   \ex(n,H_t)=M_t(n)
\]
and the extremal $H_t$-free graphs are exactly the graphs in $\mathcal{G}_{t,n}$.
\end{proposition}

\begin{proof}
For the lower bound, choose $a+b=n$ attaining $M_t(n)$.  Take $K_{A,B}$ with $|A|=a$ and $|B|=b$, and put
\[
   \left\lfloor\frac{a}{t-1}\right\rfloor K_{t-1}\cup K_{a\bmod(t-1)}
   \quad\text{inside }A,
\]
and similarly inside $B$.  This graph has $M_t(n)$ edges.  Every internal degree is at most $t-2$, so every cross edge has codegree at most $2(t-2)<3t$.  Hence no cross edge can be the base of a $3t$-page book.  If a copy of $H_t$ existed, the connected core $P_t$ would have to lie entirely in $A$ or entirely in $B$, contradicting the fact that the two internal graphs are $P_t$-free.  Thus $\ex(n,H_t)\ge M_t(n)$.

Now let $G$ be an $H_t$-free graph on $n$ vertices.  Set
\[
   R\coloneqq 3t(t-1)+t.
\]
Call an edge heavy if its codegree in $G$ is at least $R$, and light otherwise.  Let $G_H$ and $G_L$ be the spanning graphs formed by the heavy and light edges.

The graph $G_H$ is $P_t$-free.  Indeed, if $v_1\cdots v_t$ were a path in $G_H$, then each of its $t-1$ edges would have at least $R$ common neighbors in $G$.  We may choose the page vertices greedily.  Before choosing pages for the last core edge, at most
\[
   t+3t(t-2)=R-3t
\]
vertices have already been used or forbidden, so at least $3t$ fresh common neighbors remain.  This gives a copy of $H_t$, a contradiction.  Therefore
\begin{equation}\label{eq:heavy-bound}
   e(G_H)\le p_t(n).
\end{equation}

The light graph $G_L$ is $B_R$-free.  Indeed, if an edge of $G_L$ were the spine of an $R$-page book in $G_L$, then it would have at least $R$ common neighbors in $G$, and hence it would be heavy.

We next estimate $M_t(n)-p_t(n)$.  Since
\[
   p_t(N)=\frac{t-2}{2}N+O_t(1)
\]
uniformly in $N$, for every split $a+b=n$ we have
\[
   p_t(a)+p_t(b)-p_t(n)=O_t(1).
\]
Together with $ab\le\floor{n^2/4}$ and the balanced split $a=\floor{n/2}$, $b=\lceil n/2\rceil$, this gives
\[
   M_t(n)-p_t(n)=\floor{n^2/4}+O_t(1).
\]
Thus, if $e(G)\ge M_t(n)$, then by~\eqref{eq:heavy-bound},
\[
   e(G_L)=e(G)-e(G_H)\ge \floor{n^2/4}-O_t(1).
\]
Since $G_L$ is $B_R$-free, the bound~\eqref{eq:nonbip-book} implies that $G_L$ is bipartite for all sufficiently large $n$.  Indeed, if $G_L$ were non-bipartite, then
\[
   e(G_L)\le \left\lfloor \frac{(n-1)^2}{4}\right\rfloor+O_R(1)
      =\floor{n^2/4}-\frac n2+O_R(1),
\]
contradicting the lower bound above when $n$ is large enough.  Let $A\cup B$ be the bipartition of $G_L$, and write $a\coloneqq |A|$, $b\coloneqq |B|$.

All light edges go across $A\cup B$.  A heavy cross edge can only occupy a cross pair which is not already a light edge, and so
\[
   e(G_H[A,B])\le ab-e(G_L).
\]
Also $G_H[A]$ and $G_H[B]$ are $P_t$-free.  Hence
\[
   e(G_H[A])\le p_t(a)\quad\text{and}\quad e(G_H[B])\le p_t(b).
\]
It follows that
\begin{align*}
   e(G)
   &=e(G_L)+e(G_H[A,B])+e(G_H[A])+e(G_H[B])\\
   &\le e(G_L)+(ab-e(G_L))+p_t(a)+p_t(b)\\
   &\le M_t(n).
\end{align*}
This proves the upper bound.

The equality case follows from the same chain of inequalities.  If $G$ is extremal, then $(a,b)$ attains the maximum defining $M_t(n)$, all cross pairs are edges of $G$, and
\[
   e(G[A])=p_t(a)\quad\text{and}\quad e(G[B])=p_t(b).
\]
Since there are no internal light edges, $G[A]$ and $G[B]$ are subgraphs of $G_H$, and hence they are $P_t$-free.  Thus $G[A]$ and $G[B]$ are extremal $P_t$-free graphs.  Therefore $G\in\mathcal G_{t,n}$.  Conversely, every graph in $\mathcal G_{t,n}$ is $H_t$-free and has $M_t(n)$ edges, so it is extremal.
\end{proof}

The balanced order is the only case needed below.  Here the extremal graph is especially clean.
\begin{corollary}\label{cor:balanced-extremal}
Fix $t\ge6$.  For all sufficiently large $s$, the unique extremal $H_t$-free graph on $2(t-1)s$ vertices is
\[
   T^\star\coloneqq K_{(t-1)s,(t-1)s}\cup sK_{t-1}\cup sK_{t-1},
\]
where the two copies of $sK_{t-1}$ are placed inside the two sides of $K_{(t-1)s,(t-1)s}$.  In particular,
\[
   \ex(2(t-1)s,H_t)=(t-1)^2s^2+s(t-1)(t-2).
\]
Moreover, the missing pairs of $T^\star$ are exactly the pairs inside one side and in two different $K_{t-1}$-blocks, and all missing pairs are equivalent under automorphisms of $T^\star$.
\end{corollary}

\begin{proof}
Put $k\coloneqq t-1$.  We first show that the maximum defining $M_t(2ks)$ is attained only by the split $(ks,ks)$, up to exchanging the two sides.  From~\eqref{eq:path-exact}, if $N\equiv r\pmod k$ and $0\le r<k$, then
\[
   p_t(N)=\frac{k-1}{2}N-\frac{r(k-r)}2 .
\]
Consider a split $(ks+d,ks-d)$ with $d>0$, and write $d\equiv r\pmod k$, $0\le r<k$.  Relative to the balanced split, the cross term changes by
\[
   (ks+d)(ks-d)-(ks)^2=-d^2,
\]
while the two path-extremal terms change by
\[
   p_t(ks+d)+p_t(ks-d)-2p_t(ks)
   =
   \begin{cases}
      0,& r=0,\\
      -r(k-r),& r\ne0.
   \end{cases}
\]
Thus every unbalanced split gives a strictly smaller value.  Since the unique extremal $P_t$-free graph on $(t-1)s$ vertices is $sK_{t-1}$, Proposition~\ref{prop:extremal} gives the result.  The description of missing pairs is immediate.
\end{proof}

\section{The one-edge count}\label{sec:one-edge-count}

We briefly explain the comparison.  By Corollary~\ref{cor:balanced-extremal}, adding one missing edge to the extremal graph $T^\star$ joins two different $K_{t-1}$-blocks in one side and creates many core paths.  By contrast, replacing those two blocks by
\[
   K_t\cup K_{t-2}
\]
adds one edge but creates fewer core paths.  In the lower bound for $c(n,H_t)$, we count only those copies whose page vertices all lie in the opposite side; this gives the same page-selection factor $\Lambda_{t,s}$ as in the upper-bound construction and is sufficient for the comparison.

The comparison rests on two elementary counts.
\begin{lemma}\label{lem:core-counts}
Let $t\ge6$.  The following statements hold.
\begin{enumerate}
\item Let $J_t$ be obtained from two disjoint copies of $K_{t-1}$ by adding one edge between them.  Then
\[
   N_{P_t}(J_t)=2^{t-2}(t-2)!.
\]
\item The complete graph $K_t$ contains
\(
   t!/2
\)
copies of $P_t$.
\end{enumerate}
\end{lemma}

\begin{proof}
For the first assertion, every copy of $P_t$ in $J_t$ uses the added edge, say $xy$.  If the path uses $i$ further vertices in the $K_{t-1}$ containing $x$ and $t-2-i$ further vertices in the other $K_{t-1}$, then the number of choices is
\[
   \binom{t-2}{i}i!\binom{t-2}{t-2-i}(t-2-i)!.
\]
Summing over $i=0,1,\ldots,t-2$ gives
\[
   \sum_{i=0}^{t-2}\binom{t-2}{i}i!\binom{t-2}{t-2-i}(t-2-i)!
   =2^{t-2}(t-2)!.
\]
For the second assertion, a copy of $P_t$ in $K_t$ is a Hamilton path of $K_t$, so the number is $t!/2$.
\end{proof}

We now compare adding one edge to the extremal graph with the cheaper graph obtained by changing two clique blocks.
\begin{proof}[Proof of Theorem~\ref{thm:main}]
Let $n\coloneqq2(t-1)s$, where $s$ is sufficiently large.  By Corollary~\ref{cor:balanced-extremal}, the unique extremal graph is $T^\star$.  Therefore
\[
   c(n,H_t)=t_{H_t}(n,1)
   =\min_{e\notin E(T^\star)}N_{H_t}(T^\star+e).
\]

Add one missing edge $e$ to $T^\star$.  By Corollary~\ref{cor:balanced-extremal}, this edge joins two different $K_{t-1}$-blocks inside one side.  Part~(1) of Lemma~\ref{lem:core-counts} gives $2^{t-2}(t-2)!$ core copies of $P_t$ in those two blocks together with $e$.  For each such core path, every core edge has all $(t-1)s$ vertices of the opposite side as common neighbors.  Choosing disjoint $3t$-sets of page vertices in the opposite side for the $t-1$ core edges gives
\[
   \Lambda_{t,s}\coloneqq\prod_{j=0}^{t-2}\binom{(t-1)s-3tj}{3t}
\]
copies of $H_t$ for each core path.  For a fixed core path, different choices of these disjoint page sets give different subgraph copies of $H_t$.  Distinct core paths also give distinct copies, since the core edges are part of the chosen subgraph copy of $H_t$.  Extra edges in the ambient graph are irrelevant because copies are not required to be induced.  Thus every graph obtained from $T^\star$ by adding one edge contains at least
\begin{equation}\label{eq:c-lower}
   2^{t-2}(t-2)!\Lambda_{t,s}
\end{equation}
copies of $H_t$.  Hence $c(n,H_t)\ge 2^{t-2}(t-2)!\Lambda_{t,s}$.

Now construct a graph $Y_{t,s}$ with $\ex(n,H_t)+1$ edges.  Start with $T^\star$, choose two $K_{t-1}$-blocks in one side, delete their internal edges, and put $K_t\cup K_{t-2}$ on the same $2t-2$ vertices.  Since
\[
   e(K_t)+e(K_{t-2})=\binom{t}{2}+\binom{t-2}{2}
   =2\binom{t-1}{2}+1,
\]
we have $e(Y_{t,s})=\ex(n,H_t)+1$.

We count the copies of $H_t$ in $Y_{t,s}$.  Every cross edge has codegree at most
\[
   (t-1)+(t-2)=2t-3<3t,
\]
so no cross edge can be a base edge of a $3t$-page book.  Since the core path is connected, every core copy of $P_t$ must lie wholly inside one side.  The only new internal component containing a $P_t$ is the $K_t$, and by part~(2) of Lemma~\ref{lem:core-counts} it contains $t!/2$ core copies.

For each core path in this $K_t$, all $t$ vertices of the $K_t$ are already used by the core.  The other internal components in the same side are disjoint from this $K_t$, so no unused same-side vertex is adjacent to both endpoints of any core edge.  Hence all page vertices must lie in the opposite side, giving exactly $\Lambda_{t,s}$ choices.  Distinct core paths and distinct choices of page sets give distinct subgraph copies.  Therefore
\begin{equation}\label{eq:h-upper}
   h_{H_t}(n,1)\le N_{H_t}(Y_{t,s})=\frac{t!}{2}\Lambda_{t,s}.
\end{equation}
Combining~\eqref{eq:c-lower} and~\eqref{eq:h-upper},
\[
   h_{H_t}(n,1)
      \le \frac{t!}{2}\Lambda_{t,s}
      < 2^{t-2}(t-2)!\Lambda_{t,s}
      \le c(n,H_t)=t_{H_t}(n,1),
\]
where the strict inequality follows from $t(t-1)<2^{t-1}$ for every $t\ge6$.  This proves the theorem.
\end{proof}

\section*{Acknowledgments}

\begin{sloppypar}
We learned that Heng Li, Yongtao Li, and Jie Ma have independently (in an unpublished manuscript) obtained similar results, and we thank Heng Li for informing us of their work and for helpful discussions.
\end{sloppypar}
\bibliographystyle{abbrv}
\bibliography{supersaturation}

@unpublished{Edwards1975,
  author = {Edwards, C. S.},
  title = {A lower bound for the largest number of triangles with a common edge},
  note = {Unpublished manuscript},
  year = {1975}
}

@article{ErdosGallai1959,
  author = {Erd{\H{o}}s, Paul and Gallai, Tibor},
  title = {On maximal paths and circuits of graphs},
  journal = {Acta Math. Acad. Sci. Hungar.},
  fjournal = {Acta Mathematica Academiae Scientiarum Hungaricae},
  volume = {10},
  year = {1959},
  pages = {337--356}
}

@article{Erdos1955,
  author = {Erd{\H{o}}s, Paul},
  title = {Some theorems on graphs},
  journal = {Riveon Lematematika},
  volume = {9},
  year = {1955},
  pages = {13--17}
}

@article{Erdos1962,
  author = {Erd{\H{o}}s, Paul},
  title = {On a theorem of {Rademacher--Tur{\'a}n}},
  journal = {Illinois J. Math.},
  fjournal = {Illinois Journal of Mathematics},
  volume = {6},
  year = {1962},
  pages = {122--127}
}

@article{FaudreeSchelp1975,
  author = {Faudree, Ralph J. and Schelp, Richard H.},
  title = {Path {R}amsey numbers in multicolorings},
  journal = {J. Combin. Theory Ser. B},
  fjournal = {Journal of Combinatorial Theory. Series B},
  volume = {19},
  year = {1975},
  pages = {150--160}
}

@article{KhadziivanovNikiforov1979,
  author = {Khad{\v{z}}iivanov, Nikolai and Nikiforov, Vladimir},
  title = {Solution of a problem of {P}. {E}rd{\H{o}}s about the maximum number of triangles with a common edge in a graph},
  journal = {C. R. Acad. Bulgare Sci.},
  fjournal = {Comptes Rendus de l'Acad\'emie Bulgare des Sciences},
  volume = {32},
  year = {1979},
  pages = {1315--1318}
}

@article{Kopylov1977,
  author = {Kopylov, G. N.},
  title = {On maximal paths and cycles in a graph},
  journal = {Dokl. Akad. Nauk SSSR},
  fjournal = {Doklady Akademii Nauk SSSR},
  volume = {234},
  year = {1977},
  pages = {19--21}
}

@inproceedings{LovaszSimonovits1976,
  author = {Lov{\'a}sz, L{\'a}szl{\'o} and Simonovits, Mikl{\'o}s},
  title = {On the number of complete subgraphs of a graph},
  booktitle = {Proceedings of the Fifth British Combinatorial Conference (University of Aberdeen, Aberdeen, 1975)},
  series = {Congr. Numer.},
  volume = {XV},
  year = {1976},
  pages = {431--441},
  publisher = {Utilitas Mathematica},
  address = {Winnipeg}
}

@incollection{LovaszSimonovits1983,
  author = {Lov{\'a}sz, L{\'a}szl{\'o} and Simonovits, Mikl{\'o}s},
  title = {On the number of complete subgraphs of a graph. {II}},
  booktitle = {Studies in Pure Mathematics},
  publisher = {Birkh{\"a}user},
  address = {Basel},
  year = {1983},
  pages = {459--495}
}

@article{MaYuan2023,
  author = {Ma, Jie and Yuan, Long-Tu},
  title = {Supersaturation beyond color-critical graphs},
  journal = {Combinatorica},
  volume = {45},
  number = {2},
  year = {2025},
  pages = {Paper No. 18, 40},
  doi = {10.1007/s00493-025-00143-5}
}

@article{MiaoLiuVanDam2025,
  author = {Miao, Lu and Liu, Ruifang and van Dam, Edwin R.},
  title = {Tur{\'a}n number of books in non-bipartite graphs},
  journal = {J. Graph Theory},
  fjournal = {Journal of Graph Theory},
  year = {2026},
  note = {Published online 14 April 2026, {DOI}: 10.1002/jgt.70039},
  doi = {10.1002/jgt.70039}
}

@article{Mubayi2010,
  author = {Mubayi, Dhruv},
  title = {Counting substructures {I}: color critical graphs},
  journal = {Adv. Math.},
  fjournal = {Advances in Mathematics},
  volume = {225},
  number = {5},
  year = {2010},
  pages = {2731--2740}
}

@article{Mubayi2013,
  author = {Mubayi, Dhruv},
  title = {Counting substructures {II}: hypergraphs},
  journal = {Combinatorica},
  volume = {33},
  year = {2013},
  pages = {591--612}
}

@article{PikhurkoYilma2017,
  author = {Pikhurko, Oleg and Yilma, Zelealem B.},
  title = {Supersaturation problem for color-critical graphs},
  journal = {J. Combin. Theory Ser. B},
  fjournal = {Journal of Combinatorial Theory. Series B},
  volume = {123},
  year = {2017},
  pages = {148--185}
}
\end{document}